\documentclass[]{amsart}
\usepackage[utf8]{inputenc}

\usepackage{amssymb,amsmath,amsthm, bm, enumitem, mathdots, mathrsfs, mathtools, lmodern, tikz-cd, url}
\usetikzlibrary{decorations.pathmorphing}
\usepackage[cal=euler, bb=ncmbbk]{mathalpha} %bfcal, bfbb

\usepackage{microtype} % makes stuff look nicer, but sometimes creates issues with nonscalable fonts
\usepackage[textsize=tiny, color=pink]{todonotes}

\usepackage[colorlinks, breaklinks]{hyperref}
\definecolor{darkpastelblue}{rgb}{0.47, 0.62, 0.8}
\definecolor{darkpastelgreen}{rgb}{0.01, 0.75, 0.24}
\definecolor{darkpastelred}{rgb}{0.76, 0.23, 0.13}
\hypersetup{
	colorlinks,
	urlcolor={darkpastelblue!80!black},
	citecolor={darkpastelgreen!80!black},
	linkcolor={darkpastelred!80!black},
}
\def\eqref#1{\textcolor{darkpastelred}{(\ref{#1})}}

\usepackage[nameinlink]{cleveref}
\Crefformat{section}{#2\S#1#3}
\Crefmultiformat{section}{#2\S\S#1#3}{ and~#2#1#3}{, #2#1#3}{, and~#2#1#3}

\newcounter{Step}[section]
\Crefname{Step}{Step}{Steps}

\setlist[enumerate]{label={\roman*\textnormal{)}}}

\theoremstyle{plain}
\newtheorem{theorem}{Theorem}[section]
\newtheorem{lemma}[theorem]{Lemma}
\newtheorem{corollary}[theorem]{Corollary}
\newtheorem{proposition}[theorem]{Proposition}

\theoremstyle{definition}
\newtheorem{definition}[theorem]{Definition}
\newtheorem{conjecture*}{Conjecture}
\newtheorem{example}[theorem]{Example}

\theoremstyle{remark}
\newtheorem{remark}[theorem]{Remark}
\newtheorem*{remark*}{Remark}
\newtheorem*{conventions*}{Conventions}
\newtheorem*{disclaimer*}{Disclaimer}

\AddToHook{env/proposition/begin}{\crefalias{theorem}{proposition}}
\AddToHook{env/lemma/begin}{\crefalias{theorem}{lemma}}
\AddToHook{env/definition/begin}{\crefalias{theorem}{definition}}
\AddToHook{env/example/begin}{\crefalias{theorem}{example}}
\newcommand{\Abb}{\mathbb{A}}
\newcommand{\GG}{\mathbb{G}}

\newcommand{\VV}{\mathbb{V}}
\newcommand{\ZZ}{\mathbb{Z}}
\newcommand{\kk}{\mathbb{k}}

\newcommand{\bE}{\bm{E}}

\newcommand{\br}{\bm{r}}

\newcommand{\Acal}{\mathcal{A}}
\newcommand{\Bcal}{\mathcal{B}}
\newcommand{\Ecal}{\mathcal{E}}
\newcommand{\Fcal}{\mathcal{F}}
\newcommand{\Ical}{\mathcal{I}}
\newcommand{\Kcal}{\mathcal{K}}
\newcommand{\Mcal}{\mathcal{M}}
\newcommand{\Lcal}{\mathcal{L}}
\newcommand{\Ocal}{\mathcal{O}}
\newcommand{\Pcal}{\mathcal{P}}
\newcommand{\Scal}{\mathcal{S}}

\newcommand{\Ucal}{\mathcal{U}}

\newcommand{\Xcal}{\mathcal{X}}
\newcommand{\Ycal}{\mathcal{Y}}
\newcommand{\Zcal}{\mathcal{Z}}

\newcommand{\m}{\mathfrak{m}}
\newcommand{\p}{\mathfrak{p}}

\newcommand\isoto{\stackrel{\sim}{\smash{\longrightarrow}\rule{-1pt}{0.4ex}}}

\DeclareMathOperator{\op}{op}
\DeclareMathOperator{\tors}{tors}
\DeclareMathOperator{\Spec}{Spec}
\DeclareMathOperator{\rSpec}{\underline{Spec}}
\DeclareMathOperator{\Bl}{Bl}
\DeclareMathOperator{\codim}{codim}
\DeclareMathOperator{\Quot}{Quot}
\DeclareMathOperator{\Z}{Z}
\DeclareMathOperator{\rad}{rad}

\DeclareMathOperator{\Ram}{Ram}

\DeclareMathOperator{\Qcoh}{\mathsf{Qcoh}}
\DeclareMathOperator{\grMod}{\mathsf{grMod}}
\DeclareMathOperator{\Mod}{\mathsf{Mod}}
\DeclareMathOperator{\Hom}{Hom}
\DeclareMathOperator{\End}{End}
\DeclareMathOperator{\sEnd}{\mathcal{E}\! \mathit{nd}}
\DeclareMathOperator{\sHom}{\mathcal{H}\! \mathit{om}}
\DeclareMathOperator{\rProj}{\underline{Proj}}
\DeclareMathOperator{\srProj}{\underline{\mathcal{P} roj}}

\title[Azumayafication of orders]{A stacky resolution of singularities for orders by Azumaya algebras}

\author{Timothy De Deyn}
\address{Max Planck Institute for Mathematics,
	Bonn, Germany}
\curraddr{D\'{e}partement de Math\'{e}matique, Universit\'{e} Libre de Bruxelles, Boulevard du Triomphe, B-1050 Brussels, Belgium}
\email{timothy.de.deyn@ulb.be}



\begin{document}
	\begin{abstract}
		We construct a stacky resolution of singularities for certain noncommutative spaces which can be viewed as `finite noncommutative extensions' of schemes. 
		More precisely, we show that any order over a reduced separated finite type scheme over a field of characteristic zero can be resolved by an Azumaya algebra over a smooth Deligne--Mumford stack by a sequence of stacky blow-ups.
	\end{abstract}
	\maketitle
	
	\section{Introduction}
	
	Noncommutative algebraic geometry is a vast field.
	Consequently, there are many different flavours of what a `noncommutative space' is.
	In this work, we consider those that can be viewed as `finite noncommutative extensions' of schemes. 
	More precisely, we consider those spaces known as \emph{orders over schemes}, which are types of coherent sheaves of algebras over schemes; see \Cref{subsec: order} for the exact definition.
	
	In algebraic geometry, an important tool is \emph{resolution of singularities}. 
	This process allows for poorly-behaved, singular, varieties to be resolved by better-behaved, smooth, varieties.
	Our main result is a resolution of singularities for orders over varieties.
	The resolving objects are Azumaya algebras over smooth Deligne--Mumford stacks. 
	We call this process the \emph{(birational) Azumayafication of the order}.
	Let us start by stating our main theorem, after which we give some motivation.
	In the statement a \emph{stacky blow-up} refers to either a blow-up or taking a root stack, see \Cref{subsubsec:blow-ups,subsubsec:root_stacks} for recollections.
	
	\begin{theorem}\label{thm: azumification}
		Let $\Acal$ be an order over a reduced separated scheme $X$ of finite type over a field $\kk$ of characteristic zero.
		Then there exists a sheaf of Azumaya algebras $\Bcal$ over a smooth (over $\kk$) quasi-compact separated Deligne--Mumford stack $\Xcal$, together with a proper birational morphism $\pi\colon  \Xcal \to X$, obtained as the composition of stacky blow-ups, and a morphism of $\Ocal_{\Xcal}$-algebras $\pi^*\Acal\to\Bcal$ that is generically an isomorphism.
	\end{theorem} 
	\begin{remark}
		The induced morphism $\Acal\to\pi_*\Bcal$ is automatically injective, and so is necessarily an isomorphism when $\Acal$ is a \emph{maximal} order. 
		In this case, $\Acal$ is the pushforward of an Azumaya algebra over a Deligne--Mumford stack.
	\end{remark}
	
	Thus one can resolve any order by an Azumaya algebra, granted one leaves the realm of schemes and moves to Deligne--Mumford stacks.
	It is necessary to extend to algebraic stacks to obtain a birational morphism.
	If one wants to stay inside the realm of schemes, one can do this at the cost of considering alterations, i.e.\ generically finite morphisms, instead of birational morphisms.
	
	The result above should be compared to both the resolution of singularities for orders by Chan--Ingalls \cite{Chan/Ingalls:2005} and the `Morita Azumayafication' appearing in the recent paper by Faber--Ingalls--Okawa--Satriano \cite{Faber/Ingalls/Okawa/Satriano}. 
	A first distinction is that whilst the previous results are for orders over surfaces (with restrictions on the type of orders in \cite{Faber/Ingalls/Okawa/Satriano}) and allow for positive characteristic, our results are valid in arbitrary dimension and for arbitrary orders but only in characteristic zero.
		
	Additionally, in \cite{Chan/Ingalls:2005} the authors show that any order over a surface (over a ground field of characteristic prime to the degree of the order) can be resolved by a `terminal order' over a surface (which can be assumed to have simple normal crossing ramification divisor). 
	A terminal order over a surface is in general not an Azumaya algebra, but it always has global dimension two and so yields a resolution.
	The procedure in this work constructs a resolution by an Azumaya algebra and so certainly does not coincide with the resolution of Chan--Ingalls.
	However, they are closely related: Taking a suitable root stack of a terminal order over a surface (with simple normal crossing ramification divisor) yields a \emph{Morita equivalent} Azumaya algebra over a smooth Deligne--Mumford stack.
	Here Morita equivalent means that their categories of (quasi-)coherent sheaves are equivalent.
	We show this and elaborate more at the end of this paper in \Cref{sec:comparing_to_CI} .
	
	Meanwhile, in \cite{Faber/Ingalls/Okawa/Satriano} the authors show that, for any `sufficiently nice' order of global dimension two (see Definition 1.4 of \cite{Faber/Ingalls/Okawa/Satriano}) over a normal connected quasi-projective surface over an algebraically closed field of characteristic zero or at least seven, that the order is always Morita equivalent to an Azumaya algebra over a smooth Deligne--Mumford stack.
	As a consequence, see Corollary 1.6 of \cite{Faber/Ingalls/Okawa/Satriano}, they immediately obtain that any such order over a projective surface is \emph{geometrically realisable} in the sense of Orlov \cite{Orlov-dncs}, i.e.\ it admits an admissible embedding into the derived category of a smooth projective variety.
	In comparison, here we do not change the Morita class at the generic point of the order, but we also do not obtain an Azumaya algebra Morita equivalent to our original order. 
	Rather, we obtain one that is suitably birational.
	As a consequence, the result here does not directly show that orders are always geometrically realisable, in future work the current result will be used to prove so.	
	
	Let us say some words about the proof of \Cref{thm: azumification}.
	The idea is to first reduce to the case where the non-Azumaya locus is a subset of a simple normal crossings divisor.
	Then, through taking an iterated root stack, we reduce to $\Acal$ being Azumaya in codimension one; that is the non-Azumaya locus contains no divisors.
	Next, using purity, we deduce that $\Acal$ is \'etale locally the endomorphism ring of a reflexive module.
	Lastly, by a `universal flattening' in the sense of \cite{Rossi, Villamayor}, we make this reflexive module locally free, thus yielding an Azumaya algebra.
	For convenience, in the main theorem we restrict to reduced schemes of finite type over $\kk$.
	One can weaken this, e.g.\ replacing finite type by excellent, by using more elaborate resolution machinery.
	
	The structure of the note is as follows.
	We start with some preliminaries in \Cref{sec: prelim}, this includes recollections on algebraic stacks and orders, how to remove ramification of orders via root stacks, and how to blow-up modules on schemes.
	This being done, we prove \Cref{thm: azumification} in \Cref{sec: proof}.
	In the last section, \Cref{sec:comparing_to_CI}, we briefly relate our construction to that of \cite{Chan/Ingalls:2005}.
	
	\begin{conventions*}
		Throughout, $\kk$ denotes a field of characteristic zero and in \Cref{sec: proof} we work exclusively over $\kk$. 
		In particular, smoothness will always be considered over $\kk$.
	\end{conventions*}
	
	\subsection*{Acknowledgements}
	This work was started a while back during the author's PhD studies at the Vrije Universiteit Brussel under the supervision of Michel Van den Bergh.
	The author would like to thank him for generously sharing his knowledge and his help throughout the PhD. 
 	We also thank the anonymous referee for a careful reading and valuable suggestions.
	
	\subsection*{Funding}
	T.~D.~D.~is a postdoctoral researcher of the Fonds de la Recherche Scientifique -- FNRS.
	During early stages of the project the author was supported by the Research Foundation - Flanders (FWO) Ph.D.\ Fellowship fundamental research 95249 and by ERC Consolidator Grant 101001227 (MMiMMa).
	Furthermore, the author is grateful to Max Planck Institute for Mathematics in Bonn for its hospitality and financial support.
	
	\section{Preliminaries}\label{sec: prelim}
	\subsection{Background on algebraic stacks}
	Throughout we use the conventions of the Stacks Project \cite{stacks-project} concerning the definitions of algebraic stacks and
	algebraic spaces.
	
	\subsubsection{Codimension}
	Codimension for locally Noetherian algebraic stacks is defined as one would expect using a smooth atlas \cite[Definition 6.1]{Osserman}.
	Namely, let $\Zcal\subseteq\Xcal$ be a non-empty closed substack of a locally Noetherian stack, then 
	\[
	\codim_\Xcal(\Zcal)\coloneq \codim_U(U\times_\Xcal \Zcal)
	\]
	where $U\to \Xcal$ is a smooth atlas.
	We refer to \cite[\S 6]{Osserman} for details.
	
	\subsubsection{Blow-ups}\label{subsubsec:blow-ups}
	Let $\Xcal$ be an algebraic stack and $\Acal=\oplus_{n\geq 0} \Acal_n$ be a commutative quasi-coherent sheaf of graded $\Ocal_{\Xcal}$-algebras.
	Recall, see e.g.\ \cite[\S10.2.5]{Olsson} or \cite[\S14.3]{Laumon/Moret-Bailly}, the \emph{relative Proj of $\Acal$ over $\Xcal$} which is denoted $\rProj_{\Xcal}\Acal$.
	This is an algebraic stack that comes equipped with a (representable by schemes) morphism  $\rProj_{\Xcal}\Acal\to \Xcal$.
	It extends the relative Proj construction for schemes:
	for any morphism $f\colon X\to \Xcal$ from a scheme $X$ the fibre product $X \times_{\Xcal} \rProj_{\Xcal}\Acal$ is isomorphic to $\rProj_X (f^*\Acal )$.
	
	Now, let $\Ical\subseteq \Ocal_{\Xcal}$ be a quasi-coherent ideal.
	The \emph{blow-up of $\Xcal$ at $\Ical$} is the relative Proj of the Rees algebra associated to $\Ical$:
	\[
		\Bl_{\Ical}\Xcal \coloneq \rProj_{\Xcal} \left(\bigoplus_{n\geq 0}\Ical^n \right).
	\]
	In particular, for any flat morphism $f\colon X\to \Xcal$ from a scheme $X$ the fibre product $X \times_{\Xcal} 	\Bl_{\Ical}\Xcal$ is the blow-up of $X$ at the ideal $f^*\Ical$ (this is an ideal as $f$ is flat).
	\begin{remark}
		There is also a stack-theoretic Proj which is denoted $\srProj_{\Xcal}\Acal$, see e.g.\ \cite[\S10.2.7]{Olsson}.
		This Proj introduces extra stackiness, i.e.\ the natural map to $\Xcal$ is not representable.
		However, when $\Acal$ is generated in degree one, $\srProj_{\Xcal}\Acal = \rProj_{\Xcal}\Acal$; see e.g.\ \cite[Example 10.2.8 (and the proof of Proposition 10.2.9)]{Olsson} or \cite[Corollary 1.6.2]{Quek/Rydh}.
		Therefore, as the Rees algebra of an ideal is generated in degree one, it does not matter which Proj to take in the definition of the blow-up.
	\end{remark}
	
	\subsubsection{Root stacks}\label{subsubsec:root_stacks}
	The root stack construction is a way of adjoining roots of one or more divisors on a scheme (or algebraic stack), adding stacky structure along the divisors and being an isomorphism outside the union of the divisors.
	We briefly recall this here, see e.g.\ \cite[\S 3]{Cadman}, \cite[Appendix B]{AbramovichGraberVistoli}, \cite[\S 2.1]{BayerCadman}, \cite[\S 1.3]{FantechiMannNironi} or \cite[\S 3]{BerghLuntsSchnurer} for more details.
	
	Let $X$ be a scheme (or an algebraic stack), $E$ an effective Cartier divisor and $r$ a positive integer. 
	The \emph{$r$-th root stack of $X$ with respect to $E$} is defined as the pullback diagram
	\[
	\begin{tikzcd}[ampersand replacement=\&]
		{X_{r^{-1}E}} \& {[\Abb^1/\GG_m]} \\
		X \& {[\Abb^1/\GG_m]}
		\arrow[from=1-1, to=2-1]
		\arrow["{x\mapsto x^r}", from=1-2, to=2-2]
		\arrow[from=1-1, to=1-2]
		\arrow[from=2-1, to=2-2]
		\arrow["\lrcorner"{anchor=center, pos=0.125}, draw=none, from=1-1, to=2-2]
	\end{tikzcd}
	\]
	where the lower morphism $X\to [\Abb^1/\GG_m]$ corresponds to the line bundle $\Ocal_X(E)$ together with the canonical global section $\Ocal_X \to \Ocal_X(E)$.
	
	More generally, let $\bE=(E_1,\dots, E_n)$ be a tuple of effective Cartier divisors and $\br=(r_1,\dots,r_n)$ be a tuple of positive integers. 
	The \emph{$\br$-th root stack of $X$ with respect to $\bE$}, denoted $X_{\br^{-1}\bE}$, is defined as the fibre product of the root stacks $X_{r_i^{-1}E_i}$ over $X$ for $1\leq i\leq n$, i.e.\ 
	\[
	X_{\br^{-1}\bE}	\coloneq X_{r_1^{-1}E_1} \times_X \dots \times_X X_{r_n^{-1}E_n}.
	\]		
	When $X$ is smooth and $\bE$ is a simple normal crossing divisor, the root stack is itself smooth, see e.g.\ \cite[Proposition 3.9]{BerghLuntsSchnurer}. 
	
	\begin{example}\label{ex:local_root}
		For future reference, let us give a local model for the root stack.
		Suppose $X=\Spec(R)$ and $E=\VV(f)$ where $f$ is a regular element, then 
		\[
		X_{r^{-1}E} = \left[ \Spec R[x]/(x^r-f)/\mu_r \right]
		\]
		where the action is given by $\zeta \cdot x = \zeta x$.
		Furthermore, supposing for simplicity that $R$ is an integral domain, one can also write this as a $\GG_m$-quotient as follows.
		Define the \emph{scaled Rees algebra with step $r$} to be
		$\widetilde{R}\coloneq \oplus_{i\in\ZZ} f^{\lceil i/r\rceil}Rt^i\subseteq \Quot(R)[t,t^{-1}]$.
		Then
		\[
		X_{r^{-1}E} = \left[ \Spec\widetilde{R}/\GG_m \right]
		\]
		with the action coming from the $\ZZ$-grading.
		To see this, let $Y\coloneq\Spec R[x]/(x^r-f)$ and define a $\mu_r$-action on $Y\times \GG_m$ by $\zeta \cdot(y,g)\coloneq (\zeta y, \zeta^{-1}g)$.
		The action is free, so the quotient 
		\[
			Y\times^{\mu_r} \GG_m\coloneq (Y\times \GG_m)/\mu_r
		\]
		is simply given by the spectrum of the invariant ring $\left(R[x,t,t^{-1}]/(x^r - f)\right)^{\mu_r}$.
		Moreover, the natural $\GG_m$-action on $Y\times \GG_m$ (by multiplication on the second factor) commutes with the $\mu_r$-action and so descends to $Y\times^{\mu_r} \GG_m$.
		A straightforward computation shows that the invariant ring is  $\GG_m$-equivariantly isomorphic to $\widetilde{R}$.
		Consequently, 
		\[
		X_{r^{-1}E} = \left[ Y/\mu_r \right] % = \left[\frac{Y\times \GG_m}{\mu_r\times \GG_m}\right] 
		= \left[ (Y\times^{\mu_r} \GG_m)/\GG_m \right] = \left[ \Spec\widetilde{R}/\GG_m \right]
		\]
		as required; see, e.g., \cite[Lemma 3.4]{Brosnan/Reichstein/Vistoli:2007} for the second equivalence.
	\end{example}
	
	\subsection{Orders over schemes and stacks}\label{subsec: order}
	
	A quasi-coherent sheaf $\Fcal$ over a scheme $X$ is \emph{torsion-free} when, for all $x\in X$, $\Fcal_x$ is a torsion-free $\Ocal_{X,x}$-module. 
	More generally, a quasi-coherent sheaf $\Fcal$ over a (reduced locally Noetherian\footnote{Checking that torsion-freeness is flat local is subtle in general, but relatively straightforward when reduced and locally Noetherian (see also the beginning of \Cref{subsec:blow-up_module}).}) algebraic stack $\Xcal$ is \emph{torsion-free} when $p^*\Fcal$ is torsion-free for some smooth atlas $p\colon U\to \Xcal$ (and hence all).
	
	An order $\Acal$ over a (reduced locally Noetherian) algebraic stack $\Xcal$ is a torsion-free coherent sheaf of algebras that is generically Azumaya, i.e.\ there exists a dense open substack $\Ucal\subset\Xcal$ such that the restricted sheaf of algebras $\Acal|_{\Ucal}$ is Azumaya.
	More concretely, over an integral scheme with generic point $\eta$ this means that there are no non-zero sections which become zero in the stalk at $\eta$ and that $\Acal_\eta$ is a central simple $\kk(\eta)$-algebra.
	
	To define maximal orders we restrict for simplicity to the case where $X$ is an integral Noetherian scheme with generic point $\eta$.
	Since the property of being maximal is not closed under \'etale extensions, the notion is not that useful over algebraic stacks.
	Let $A$ be a central simple algebra over $\kk(\eta)$.
	An order $\Acal$ over $X$ is called an \emph{order in $A$} if $\Acal_\eta= A$.
	As the natural morphism $\Acal\to \underline{A}$ (= constant sheaf with value $A$) is injective, since $\Acal$ is torsion-free, orders in $A$ can be viewed as `subalgebras' of $A$.
	A \emph{maximal order in $A$ over $X$} is an order which is not properly contained in any other order in $A$ over $X$.
	We say an order $\Acal$ is \emph{maximal} if it is maximal in $\Acal_\eta$.
	Being maximal is Zariski and stalk local and any Azumaya algebra over a normal scheme is maximal.  
	Under quite general hypotheses on $X$ (and $A$) maximal orders containing a given order exist, see \cite[Theorem 1.2]{Yu}.
	Taking a maximal order is akin to taking the integral closure in commutative algebra.
	However, it should be noted that maximal orders are generally not unique.

	Lastly, recall that an order (again over a scheme for simplicity) is called \emph{tame}, if it is reflexive as $\mathcal{O}_X$-module and its stalks at codimension one points are hereditary rings, i.e.\ they have global dimension one.
	Maximal orders are always tame over normal schemes.
	
	\subsection{Ramification}
	We define the ramification data of tame orders.
	This measures the failure of the order to be Azumaya.
	
	\subsubsection{Local}
	Let $(R,\m)$ be a DVR and denote its residue field $\kappa\coloneq R/\m$.
	Assume for simplicity that $\kappa$ is perfect. 
	
	Consider a tame, i.e.\ hereditary, $R$-order $\Lambda$ (in a central simple algebra).
	Let $\kappa'\coloneq\Z(\Lambda/\rad\Lambda)$, where $\rad$ denotes the Jacobson radical; it is a product of cyclic field extension of $\kappa$ (this follows from the structure theorem of hereditary orders over complete DVRs).
	Define the \emph{ramification index of $\Lambda$} to be the integer $r\coloneq\dim_\kappa\kappa'$. 
	One has that $r$ is the minimal integer such that $(\rad\Lambda)^r = {\m}{\Lambda}$. 
	The ramification index measures the failure of $\Lambda$ to be Azumaya, i.e.\ of $\Lambda/\m\Lambda$ being a central simple algebra, as one can show that $\Lambda$ is Azumaya if and only if $r=1$.
	
	\subsubsection{Global}\label{subsubsec: global ram}
	Let $\Acal$ be a tame order over a normal Noetherian scheme.
	Assume for simplicity that the residue fields at codimension one points are perfect.
	Let $\Ram(\Acal)$ denote the ramification locus of $\Acal$, i.e.\ the non-Azumaya locus 
	\[
	\{x\in X \mid \Acal_x\text{ is not an Azumaya }\Ocal_{X,x}\text{-algebra}\}.
	\]
	As the condition that $\Acal$ is Azumaya, i.e.\ is locally free and $\Acal\otimes_{\Ocal_X}\Acal^{\op}\to \sEnd_{\Ocal_X}(\Acal)$ is an isomorphism, is clearly an open condition (see e.g.\ \cite[\href{https://stacks.math.columbia.edu/tag/0B8J}{Tag 0B8J} and \href{https://stacks.math.columbia.edu/tag/01Y3}{Tag 01Y3}]{stacks-project}), this is a proper closed subset.
	We define the \emph{ramification data of $\Acal$} to be the divisors $\bE=(E_1,\dots, E_n)$ contained in $\Ram(\Acal)$ and the corresponding ramification indices $\br=(r_1,\dots,r_n)$ at their generic points, i.e.\ $r_i$ is the ramification index of $\Acal_{E_i}$ as $\Ocal_{X,E_i}$-order, where by abuse of notation we let $E_i$ also denote its generic point.
	
	\subsection{Rooting out ramification}
	We show how one can (birationally) get rid of ramification in height one by taking root stacks. 
	In fact, taking root stacks is not strictly needed for removing the ramification. 
	However, it is needed if one wants to remain in the same birationality class.
	Pertinent references where we pull content from are \cite{LeBruyn/VandenBergh/VanOystaeyen}, \cite[Chapter 5]{Reiten/VandenBergh}, \cite[\S\S2\&3]{Faber/Ingalls/Okawa/Satriano} and \cite[\S2.1]{Baumann/Belmand/VanGarderen}.
	Although an important difference between the last three references and what we do below is that we do a $\ZZ^n$-graded version instead of merely a $\ZZ$-graded version.
	This is required to obtain a root stack of the form $X_{\br^{-1}\bE}$ instead of $X_{\operatorname{lcm}(r_1,\dots r_n)^{-1}(E_1+\dots+E_n)}$.

	\subsubsection{(Affine) local}
	Suppose $R$ is a normal Noetherian domain containing the field $\kk$ and denote the fraction field of $R$ by $K$.
	Let $\Lambda$ be a tame $R$-order in a central simple $K$-algebra $A=K\Lambda$. 
	
	We start by recalling some terminology: a finitely generated $\Lambda$-subbimodule $I\subset A$ is called a \emph{fractional ideal (of $\Lambda$)} if it satisfies $KI=A$; moreover, it is called a \emph{divisorial ideal} if it is additionally reflexive as $R$-module and $I_\p$ is invertible for every prime $\p\subset R$ of height one.
	For such a divisorial ideal, define the $n$th symbolic power
	\[
	I^{(n)}\coloneq \begin{cases}
		(I^{ n})^{**} &\text{for } n> 0, \\
		\Lambda &\text{for } n= 0, \\
		[\Lambda:I]^{(-n)} &\text{for } n< 0, 
	\end{cases}
	\]
	where $[\Lambda:I]\coloneq\{ a\in A\mid aI\subseteq \Lambda\}$ and $(-)^*$ denotes the $R$-dual $\Hom_R(-,R)$.
	Divisorial ideals form a group under $I\cdot J\coloneq(I\otimes_R J)^{**}=(IJ)^{**}$.
	
	Next, denote the ramification data of $\Lambda$ (as defined in the previous subsection) $\bE=(E_1,\dots, E_n)$ and  $\br=(r_1,\dots,r_n)$; denote the generic point of $E_i$ by $\p_i$.
	Let
	\[
	P_i\coloneq \rad \Lambda_{\p_i}\cap \Lambda \subset \Lambda,
	\]
	this is a divisorial ideal containing $\p_i\Lambda$ with $r_i$ the minimal integer such that $P_i^{(r_i)}
	= (\p_i\Lambda)^{**}$.
	Moreover, it is the unique divisorial ideal satisfying the following at height one primes
	\begin{equation}\label{eq: def_prop_Pi}
		(P_i)_{\p}\coloneq \begin{cases}
			\rad \Lambda_{\p_i} &\text{for } \p=\p_i, \\
			\Lambda &\text{for } \p\neq\p_i. 
		\end{cases}
	\end{equation}

	Define the \emph{$\ZZ^n$-graded ramification Rees algebra}
	\begin{align*}
		\widetilde{\Lambda}&\coloneq \bigoplus_{(i_1,\dots, i_n)\in \ZZ^n}  P_1^{(i_1)} \cdot P_2^{(i_2)}  \cdots P_n^{(i_n)} t_1^{i_1}t_2^{i_2}\dots t_n^{i_n}\\ &\subset \Lambda[t_1,t_1^{-1},t_2,t_2^{-1},\dots,t_n,t_n^{-1}].
	\end{align*}
	Note that this is indeed an algebra as 
	$P_i\cdot P_j=P_i\cap P_j=P_j\cdot P_i$ for $i\neq j$ which may be verified by
	checking at height one primes.
	
	The main upshot of this construction is that $\widetilde{\Lambda}$ is a (graded) maximal order over its centre by \cite[Theorem II.4.38 and Remark II.4.40]{LeBruyn/VandenBergh/VanOystaeyen}, and is a \emph{reflexive Azumaya algebra} over its centre by combining \cite[Theorem II.4.38 and Remark II.4.40]{LeBruyn/VandenBergh/VanOystaeyen} with \cite[Theorem II.4.24]{LeBruyn/VandenBergh/VanOystaeyen}; here we are making use of the assumption that $\kk$ is of characteristic zero.
	Being reflexive Azumaya means that it is reflexive as module over its centre and is Azumaya at codimension one points (instead of merely hereditary as $\Lambda$ was).
	As reflexive Azumaya algebras are automatically orders (and maximal if the centre is normal), we sometimes say reflexive Azumaya order to emphasise this.
	The centre of $\widetilde{\Lambda}$ can be made explicit, it is the so-called \emph{scaled
	Rees ring with step $(r_1,\dots, r_n)$}; the $(i_1,\dots, i_n)$th graded piece has coefficients in
	\[
	\p_1^{(\lceil i_1/r_1\rceil)} \cdot \p_2^{(\lceil i_2/r_2\rceil)}  \cdots \p_n^{(\lceil i_n/r_n\rceil)}.
	\]
	When the $\p_i$'s are principal, taking the stacky quotient of the centre by the natural $\GG_m^n$-action gives a local model for the iterated root stack, cf.\ \Cref{ex:local_root}.
	\begin{remark}
		As a brief aside let us explain the terminology `reflexive Azumaya'.
		An $R$-algebra $\Lambda$ is reflexive Azumaya (i.e.\ reflexive as $R$-module and Azumaya at the height one primes of $R$) if and only if $\Lambda$ is reflexive as $R$-module and $(\Lambda \otimes_R \Lambda^{\op})^{**}\to \End_R(\Lambda)$ is an isomorphism.
		Thus, $\Lambda$ is an `Azumaya object' in the closed symmetric monoidal category of (finitely generated) reflexive modules $(\mathsf{ref}(R), (-\otimes_R-)^{**})$. 
	\end{remark}

	\subsubsection{Global}
		
		Let us now give the global version of the above.
		Fix a normal integral Noetherian scheme $X$ over $\kk$ and suppose $\Acal$ is a tame order (in a central simple $\kk(X)$-algebra $A$) with ramification data $\bE=(E_1,\dots, E_n)$ and  $\br=(r_1,\dots,r_n)$.
		
		By \Cref{eq: def_prop_Pi}, and the fact that for a reflexive module $M$ over a domain $M=\cap_{\operatorname{ht}\p=1}M_\p$, it is clear that the affine local $P_i$'s glue giving (divisorial ideal) sheaves $\Pcal_i$ uniquely defined by the analogous stalk local property as \Cref{eq: def_prop_Pi}.
		Therefore, we can define the (global) \emph{$\ZZ^n$-graded ramification Rees algebra}
		\begin{align*}
			\widetilde{\Acal}&\coloneq \bigoplus_{(i_1,\dots, i_n)\in \ZZ^n}  \Pcal_1^{(i_1)} \cdot \Pcal_2^{(i_2)}  \cdots \Pcal_n^{(i_n)} t_1^{i_1}t_2^{i_2}\dots t_n^{i_n}\\ 
			&\subset \underline{A}[t_1,t_1^{-1},t_2,t_2^{-1},\dots,t_n,t_n^{-1}]
		\end{align*}
		where $\underline{A}$ is the constant sheaf on $A=\Acal\otimes_{\Ocal_X} \kk(X)$.
		This is a maximal order and a reflexive Azumuya algebra over its centre as this can be checked affine locally.
		Moreover, the centre $\Z(\widetilde{\Acal})$ of $\widetilde{\Acal}$ admits a natural $\GG_m^n$-action corresponding to its $\ZZ^n$-grading.
		
		The following is the main result of this subsection.
		
		\begin{proposition}\label{prop:rooting}
			Let $(X,\Acal)$ and $\widetilde{\Acal}$ be as above and suppose $\bE=(E_1,\dots, E_n)$ is a tuple of effective Cartier divisors.
			Then
			\[	
			\left[\rSpec_X  \Z(\widetilde{\Acal})  / \GG_m^n\right] = X_{\br^{-1}\bE}
			\]
			and under this identification $\widetilde{\Acal}$ corresponds to a reflexive Azumaya order $\Bcal$.
		\end{proposition}
		\begin{proof}
			The identification of the root stack with the quotient stack can be checked locally on $X$.
			So, we may assume $X$ is affine and the divisors $\bE$ are principal.
			In this case it follows from \Cref{ex:local_root} and the description of $\Z(\widetilde{\Acal})$ as the scaled Rees ring with step $\br$.
			
			To see the second statement, as $\widetilde{\Acal}$ is a $\ZZ^n$-graded $\Z(\widetilde{\Acal})$-algebra it corresponds to an algebra $\Bcal$ over $\left[\rSpec_X \Z(\widetilde{\Acal}) /\GG_m^n\right] = X_{\br^{-1}\bE}$.
			That this algebra is a reflexive Azumaya order can be seen, e.g., by using the characterisation of being an `Azumaya object' in the category of coherent reflexive sheaves (or simply by using the atlas $\rSpec_X \Z(\widetilde{\Acal})  \to X_{\br^{-1}\bE}$).
		\end{proof}
		
		\subsection{Blowing up modules}\label{subsec:blow-up_module}

		In general, even for Noetherian schemes, the sheaf of meromorphic functions $\Kcal_X$ need not be quasi-coherent \cite{Kleiman}.
		Therefore, defining a `correct' notion of quasi-coherent torsion submodule (and getting this to behave well with respect to pulling back along flat morphisms) is subtle.
		In the reduced locally Noetherian setting these issues disappear, so we restrict to those schemes in this subsection.
		
		In fact, for a reduced locally Noetherian scheme $X$ and a (quasi-)coherent module $\Mcal$, the torsion submodule 
		\[
		\tors(\Mcal) \coloneq \ker \left(\Mcal\to \Mcal\otimes_{\Ocal_X}\Kcal_X \right) \subseteq \Mcal
		\]	
		is (quasi-)coherent and has as sections over an open subset $U\subseteq X$ those sections in $\Mcal(U)$ that become zero in the stalks at the generic points of the irreducible components of $U$.
		One can see this using, e.g., \cite[\href{https://stacks.math.columbia.edu/tag/02OW}{Tag 02OW}]{stacks-project}.
		
		\begin{definition}\label{def:flattening}
			Let $X$ be a reduced locally Noetherian scheme and $\Mcal$ a coherent $\Ocal_X$-module of constant generic rank $r$, i.e.\ $\Mcal$ is locally free of rank $r$ on a dense open subset of $X$.
			A morphism $f\colon Y\to X$ (from a reduced locally Noetherian scheme $Y$) is called \emph{a flattening of $\Mcal$} if 
			\[
			f^\flat(\Mcal)\coloneq f^*(\Mcal)/\tors(f^*(\Mcal))
			\]
			is locally free of rank $r$.
			The flattening $f$ is called \emph{universal} if every other flattening of $\Mcal$ factors uniquely through $f$.
			The universal flattening of $\Mcal$ exists (see below) and is called the \emph{blow-up of $X$ at $\Mcal$}; it (or rather, its source) is denoted by $\Bl_\Mcal(X)$.
		\end{definition}
		\begin{remark}\hfill
			\begin{enumerate}
				\item A finitely presented flat module is locally free.
				So, requiring $f^\flat(\Mcal)$ to be locally free (of finite rank) is the same as requiring it to be flat; this explains the name flattening.
				\item As we assume $X$ reduced, the local rings at the generic points of its irreducible components are fields.
				Therefore, any module is generically locally free.
				The real requirement in the definition above is that the rank has to be constant equal to $r$.
				\item The notion of flattening makes sense for more general locally Noetherian schemes (as long as the sheaf of meromorphic functions is quasi-coherent), but the conditions on $\Mcal$ are a bit more subtle.
				One requires $\Mcal \otimes_{\Ocal_X} \Kcal_X = (\Kcal_X)^r$ for some $r$.
				\item The rank condition is necessary for obtaining a universal property.
				For example, let $R=k[x,y]$ be the polynomial ring in two variables and $I=(x,y)$ the maximal ideal of the origin (ideals, of integral schemes, are always of constant generic rank 1).
				Consider $f\colon \Spec(R/I)\hookrightarrow \Spec R$, then $f^\flat(I)=k^2$. Showing that the rank can increase. 
				For the blow-up $g\colon \Bl_I(\Spec R)\to \Spec R$ we have that $g^\flat(I)=g^{-1}(I)\cdot\Ocal$ is locally free of rank one (one way to see the equality is using that $g$ is a dominant morphism between integral scheme).			
				Suppose there existed a universal flattening $h\colon X\to\Spec R$ without fixing the rank, i.e.\ we only require $\mathcal{I}\coloneq h^\flat(I)$ to be locally free. 
				Then for $f$ to factor through $h$ we need $\mathcal{I}$ locally free of rank two, but for $g$ to factor through $h$ we need $\mathcal{I}$ locally free of rank one.
			\end{enumerate}
		\end{remark}
		
		The universal flattening exists, and thus is unique up to unique isomorphism.
		It can be constructed as a certain subscheme of a Quot scheme, or by blowing up a specific coherent ideal (which is the fitting ideal of a module constructed from $\Mcal$); see \cite[Theorem 3.5]{Rossi}, \cite[Theorem 2]{Riemenschneider} or \cite[Theorem 3.3]{Villamayor} (and the paragraph afterwards).
		The latter construction shows that the blow-up is indeed a reduced locally Noetherian scheme, and is quasi-compact resp.\ integral when $X$ is.
		Moreover, it follows that the blow-up is projective and birational in the sense of \cite[\href{https://stacks.math.columbia.edu/tag/01RO}{Tag 01RO}]{stacks-project}.
		Indeed, by the universal property, and \Cref{lem: compat flat smooth cov} below, the universal flattening is an isomorphism over the open subset where $\Mcal$ is flat, and this locus is dense by assumption.
		That the inverse image of this open subset contains the generic points of all the irreducible components of the blow-up follows by \cite[\href{https://stacks.math.columbia.edu/tag/0BFM}{Tag 0BFM}]{stacks-project}.
		
		We finish this section by noting some lemmas for future use.
		\begin{lemma}\label{lem: pb flat morph}
			For any flat morphism $f\colon Y\to X$ between reduced locally Noetherian schemes 
			and quasi-coherent module $\Mcal$ we have
			\[
			f^*(\tors(\Mcal))=\tors(f^*(\Mcal)).
			\]
		\end{lemma}
		\begin{proof}
			That $f^*(\tors(\Mcal))\subseteq\tors(f^*(\Mcal))$ follows as torsion sections are exactly those sections that vanish at the generic points $\eta_i$ of the irreducible components of $Y$ and 
			\[
			f^*(\tors(\Mcal))_{\eta_i}=\tors(\Mcal)_{f(\eta_i)} \otimes_{\Ocal_{X,f(\eta_i)}} \Ocal_{Y,\eta_i}=0
			\]
			as $f(\eta_i)$ is the generic point of an irreducible component of $X$ (by \cite[Corollaire 2.3.4]{EGAIV2}).
			
			For the opposite inclusion note that $f^*(\Mcal/\tors(\Mcal))$ is torsion free (pullback along flat morphisms preserves torsion freeness).
			Therefore $\tors(f^*(\Mcal))$ becomes zero in $f^*(\Mcal/\tors(\Mcal))=f^*(\Mcal)/f^*(\tors(\Mcal))$, i.e.\ $\tors(f^*(\Mcal))\subseteq f^*(\tors(\Mcal))$. 
		\end{proof}
		
		\begin{lemma}[Blowing up commutes with flat base change]\label{lem: compat flat smooth cov}
			Let $\Mcal$ be a coherent module of constant generic rank over a reduced locally Noetherian scheme $X$ and $f\colon Y\to X$ a flat morphism with $Y$ a reduced locally Noetherian scheme.
			Then the square
			\[
			\begin{tikzcd}
				{\Bl_{f^*\Mcal}(Y)} & {\Bl_{\Mcal}(X)} \\
				Y & {X\rlap{ ,}}
				\arrow[from=1-1, to=1-2]
				\arrow[from=1-2, to=2-2]
				\arrow["q", from=1-1, to=2-1]
				\arrow["f", from=2-1, to=2-2]
				\arrow[phantom, "\lrcorner", very near start,, from=1-1, to=2-2]
			\end{tikzcd}
			\]
			with vertical morphisms being blow-ups and upper horizontal morphism induced by the universal property of the blow-up, is cartesian.
			In particular, this holds for $f\colon Y\to X$ any smooth cover.
		\end{lemma}
		\begin{proof}
			As $f$ is flat the generic point of any irreducible component of $Y$ gets mapped to the generic point of an irreducible component of $X$ (by \cite[Corollaire 2.3.2)]{EGAIV2}).
			Consequently $f^*\Mcal$ is of the same rank as $\Mcal$, and thus
			\[
			(fq)^\flat(\Mcal)=q^*f^*\Mcal/\tors(q^*f^*\Mcal)=q^\flat (f^*\Mcal)
			\]
			is flat of same rank as $\Mcal$ by definition of $q$. 
			Thus, by the universal property, there exists a morphism $\Bl_{f^*\Mcal}(Y)\to \Bl_{\Mcal}(X)$ making the diagram commute.
			That the diagram is cartesian follows from the universal properties of the blow-up and the fibre product using \Cref{lem: pb flat morph}, since, labelling the morphisms
			\[
			\begin{tikzcd}
				{Y\times_X\Bl_{\Mcal}(X)} & {\Bl_{\Mcal}(X)} \\
				Y & {X\rlap{ ,}}
				\arrow["g", from=1-1, to=1-2]
				\arrow["p", from=1-2, to=2-2]
				\arrow["q'", from=1-1, to=2-1]
				\arrow["f", from=2-1, to=2-2]
				\arrow[phantom, "\lrcorner", very near start,, from=1-1, to=2-2]
			\end{tikzcd}
			\]
			we have that
			\begin{multline*}
				(q')^\flat(f^*\Mcal) = (q')^*f^*\Mcal/\tors((q')^*f^*\Mcal) = g^*p^*\Mcal/\tors(g^*p^*\Mcal)\\ 
				= g^*p^*\Mcal/g^*\tors(p^*\Mcal) = g^*(p^*\Mcal/\tors(p^*\Mcal))= g^* (p^\flat\Mcal)
			\end{multline*}
			is flat of same rank as $\Mcal$.
			
			For the second statement, note that $Y$ is locally Noetherian by fppf descent \cite[\href{https://stacks.math.columbia.edu/tag/034C}{Tag 034C}]{stacks-project} and reduced by smooth descent \cite[\href{https://stacks.math.columbia.edu/tag/034E}{Tag 034E}]{stacks-project}.
		\end{proof}
		
		\begin{remark}			
			A straightforward way to extend \Cref{def:flattening} to algebraic stacks is by restricting to morphisms representable (by schemes).
			Then everything is determined locally by what happens in the category of schemes (and algebraic spaces). 
			One can use \Cref{lem: compat flat smooth cov} to the show the existence of the universal \emph{representable} flattening of algebraic stacks by extending from schemes to algebraic spaces (using a presentation in terms of schemes) and then to algebraic stacks (using a presentation in terms of algebraic spaces); cf.\ also with \Cref{step3} of the proof of \Cref{thm: azumification} below.
		\end{remark}
		
		\section{Proof of \texorpdfstring{\Cref{thm: azumification}}{Theorem 1.1}}\label{sec: proof}
		Suppose $\Acal$ is an order over a scheme $X$ that is locally free as $\Ocal_X$-module.
		In this case being Azumaya is equivalent to being Azumaya in codimension one.
		Therefore the strategy of the proof is to achieve the following:
		\begin{enumerate}
			\item\label{item: 1} make $\Acal$ Azumaya in codimension one,
			\item\label{item: 2} make $\Acal$ locally free.
		\end{enumerate}
		For \ref{item: 1} we will use root stacks and for \ref{item: 2} we will use blow-ups.
		Of course, the order and the exact way of doing this is important and so requires some care.
		
		\subsection*{Step 1}
		\refstepcounter{Step}\label[Step]{step1}
		Let $(X,\Acal)$ be as in \Cref{thm: azumification}.
		First resolve $X$ and replacing $\Acal$ by
		\[
		\pi^\flat\Acal\coloneq\pi^*\Acal/\tors(\pi^*\Acal), 
		\]
		where $\pi$ denotes the resolution (the pullback will generally only remain torsion free when $\pi$ is flat).
		We may thus assume, by looking per connected component, that $X$ is a smooth variety.
		Moreover, as the morphism $\pi$ is birational, and hence induces an isomorphism at the generic points, $\pi^\flat\Acal$ remains an order and the quotient map $\pi^*\Acal\to \pi^\flat\Acal$ is generically an isomorphism. 
		
		Next by an embedded resolution of singularities \cite[Corollary 3]{Hironaka}, see also e.g.\ \cite[Theorem 1.10]{BierstoneMilman} or \cite[Theorem 1.0.1]{Wlodarczyk}, there exists a further resolution of singularities $\pi\colon X'\to X$ with $\pi^{-1}(\Ram(\Acal))$ a simple normal crossing (snc) divisor.
		
		We have $\Ram(\pi^\flat\Acal)\subseteq\pi^{-1}(\Ram(\Acal))$ as being Azumaya is closed under base change.
		Thus, by replacing the pair $(X,\Acal)$ by $(X',\pi^\flat\Acal)$, and again taking the quotient $\pi^*(\Acal)\to \pi^\flat(\Acal)$ for the morphism required in the theorem, we reduce to the case that $X$ is a smooth variety and the ramification locus of $\Acal$ is contained in a snc divisor.
		Furthermore, by replacing the order $\Acal$ by a maximal order containing it we may assume $\Acal$ is maximal.
		Lastly, note that both resolutions in this step were compositions of blow-ups.
		
		\subsection*{Step 2}
		\refstepcounter{Step}\label[Step]{step2}
		Now suppose $X$ is a smooth variety and $\Acal$ is a maximal order with ramification locus $\Ram(\Acal)$ contained in a snc divisor.
		We use a root stack to get rid of the ramification divisor of $\Acal$, i.e.\ all the codimension one points in $\Ram(\Acal)$.
		
		Let $\bE=(E_1,\dots,E_n)$ and $\br=(r_1,\dots,r_n)$ be the ramification data of $\Acal$, see \Cref{subsubsec: global ram}.
		Put
		\[
		\begin{tikzcd}
			\Xcal\coloneq X_{\br^{-1}\bE} \arrow[r, "\pi"] & X
		\end{tikzcd}
		\]
		the $\br$-th root stack with respect to $\bE$ and let $\Bcal$ denote the reflexive Azumaya order from \Cref{prop:rooting}; and note that there is a induced generic isomorphism $\pi^*\Acal\to \Bcal$ given by multiplication $\Acal\otimes_{\mathcal{O}_X}\Z(\widetilde{\Acal})\to \widetilde{\Acal}$ (notation from \Cref{prop:rooting}).
		Moreover, since the divisor $\bE$ is snc, it follows that $\Xcal$ is smooth.
		
		Replacing the pair $(X,\Acal)$ by $(\Xcal,\Bcal)$ we have reduced to the case where the order $\Acal$ is reflexive and Azumaya in codimension one over a smooth quasi-compact separated Deligne--Mumford stack $\Xcal$ (with trivial generic stabiliser).
		
		\subsection*{Step 3}
		\refstepcounter{Step}\label[Step]{step3}
		We have reduced to the setting where $\Xcal$ is a smooth quasi-compact separated Deligne--Mumford stack (with trivial generic stabiliser) and the order $\Acal$ is reflexive and Azumaya in codimension one.
		Let $\Zcal\coloneq\Ram(\Acal)$ be the ramification locus, then by assumption $\Zcal$ has codimension at least $2$ in $\Xcal$.
		Denote $\Ucal\coloneq\Xcal\setminus\Zcal$ the complement.
		By purity for smooth Deligne--Mumford stacks, see e.g.\ \cite[Lemma 5.1]{LoughranSantens}, there exists an $\alpha \in H^2_{\acute{e}t}(\Xcal,\GG_m)$ restricting to $\Acal|_{\Ucal}\in H^2_{\acute{e}t}(\Ucal,\GG_m)$.
		Furthermore, by \cite[\href{https://stacks.math.columbia.edu/tag/01FW}{Tag 01FW}]{stacks-project} there exists a splitting of $\alpha$, i.e.\ an \'etale cover $v\colon V\to \Xcal$ by a (smooth) scheme $V$ with $\alpha|_V$ zero in $H^2_{\acute{e}t}(V,\GG_m)$.  
		Consequently $\Acal|_{V\times_{\Xcal}\Ucal}=\alpha|_{V\times_{\Xcal}\Ucal}=0$ and thus is of the form $\sEnd_{V\times_{\Xcal}\Ucal}(\Ecal)$ for some locally free sheaf $\Ecal$.
		Now $j\colon V\times_{\Xcal}\Ucal\to V$ is an open immersion whose complement has codimension at least two.
		As $\Acal$ is reflexive this implies that $\Acal|_V=\sEnd_V(j_*\Ecal)$ (use e.g.\ \cite[\href{https://stacks.math.columbia.edu/tag/0EBJ}{Tag 0EBJ}]{stacks-project} and that $V$ is the disjoint union of integral schemes).
		Thus, $\Acal$ is \'etale locally given by the endomorphism sheaf of a reflexive module.
		The idea is now to blow-up $\Mcal\coloneq j_*\Ecal$ on $V$, and then, by keeping track of the descent data, glue the endomorphism sheaf of its flattening to an Azumaya algebra over a Deligne--Mumford stack. 
		%		As we have Lemma \ref{lem: pb flat morph} everything works out nicely.
		
		Denote by $p_i\colon V^2\coloneq V\times_\Xcal V\to V$ the two projections from the fibre product, and by $p_{ij}\colon V^3\coloneq V\times_\Xcal V\times_\Xcal V\to V^2$ and $p_{i}\colon V^3\to V$ the various projections from the triple fibre product.
		Note that $V^2$ is a scheme (and thus also $V^3$) as the diagonal of $\Xcal$ is finite (and hence representable by schemes).
		The descent data of $\Acal$ consists of an isomorphism $\alpha\colon p_1^*(\Acal|_V)\isoto p_2^*(\Acal|_V)$ over $V^2$ satisfying the cocycle condition
		% https://q.uiver.app/?q=WzAsNixbMSwwLCJwX3sxMn1eKnBfMVxcQWNhbHxfViJdLFswLDEsInBfezEzfV4qcF8xXFxBY2FsfF9WIl0sWzEsMiwicF97MTN9XipwXzJcXEFjYWx8X1YiXSxbMiwyLCJwX3syM31eKnBfMlxcQWNhbHxfViJdLFszLDEsInBfezIzfV4qcF8xXFxBY2FsfF9WIl0sWzIsMCwicF97MTJ9XipwXzJcXEFjYWx8X1YiXSxbMSwwLCIiLDAseyJzdHlsZSI6eyJib2R5Ijp7Im5hbWUiOiJub25lIn19fV0sWzAsNSwicF97MTJ9XipcXGFscGhhIl0sWzUsNCwiIiwwLHsic3R5bGUiOnsiYm9keSI6eyJuYW1lIjoibm9uZSJ9fX1dLFs0LDMsInBfezIzfV4qXFxhbHBoYSJdLFsxLDIsInBfezEzfV4qXFxhbHBoYSIsMl0sWzIsMywiIiwyLHsic3R5bGUiOnsiYm9keSI6eyJuYW1lIjoibm9uZSJ9fX1dXQ==
		\[
		\begin{tikzcd}[column sep= -1em, row sep= 3em]
			& {p_{12}^*p_1^*\Acal|_V} &[3.5em] {p_{12}^*p_2^*\Acal|_V} \\
			{p_{13}^*p_1^*\Acal|_V} &&& {p_{23}^*p_1^*\Acal|_V} \\
			& {p_{13}^*p_2^*\Acal|_V} & {p_{23}^*p_2^*\Acal|_V}
			\arrow[equal, from=2-1, to=1-2]
			\arrow["{p_{12}^*\alpha}", from=1-2, to=1-3]
			\arrow[equal, from=1-3, to=2-4]
			\arrow["{p_{23}^*\alpha}", from=2-4, to=3-3]
			\arrow["{p_{13}^*\alpha}"', from=2-1, to=3-2]
			\arrow[equal, from=3-2, to=3-3]
		\end{tikzcd}
		\]
		over $V^3$.
		
		Let us now translate this descent data to descent data in terms of $\Mcal$ instead of $\Acal|_V= \sEnd_V(\Mcal)$.
		The isomorphism $\alpha$ induces an isomorphism 
		\begin{equation}\label{eq: gluing data}
			\begin{tikzcd}
				\sEnd_{V^2}(p_1^*\Mcal)=p_1^*{\sEnd}_V(\Mcal)\arrow[r, "\alpha"] &p_2^*{\sEnd}_V(\Mcal)=\sEnd_{V^2}(p_2^*\Mcal)
			\end{tikzcd}
		\end{equation}
		which in turn is induced by an isomorphism $\gamma\colon \Lcal\otimes_{V^2}p_1^*\Mcal\isoto p_2^*\Mcal$ with $\Lcal$ an invertible module over $V^2$. 
		Indeed, by reflexive Morita theory, taking 
		\[
		\Lcal\coloneq\sHom_{p_1^*(\Acal|_V)}(p_1^*\Mcal,p_2^*\Mcal),
		\]
		where the $p_1^*(\Acal|_V)$-action on $p_2^*\Mcal$ comes from $\alpha$, and taking $\gamma$ to be the evaluation does the trick. 
		Here, we used the fact that $V^2$ is smooth, hence locally factorial, to conclude that $\Lcal$ is invertible instead of merely reflexive of rank one.
		
		The cocycle condition translates into the following condition on $\gamma$.
		There exists an isomorphism
		\[
		\phi\colon  p_{23}^*\Lcal\otimes_{V^3}p_{12}^*\Lcal\isoto p_{12}^*\Lcal 
		\]
		satisfying
		\[
		% https://q.uiver.app/?q=WzAsNCxbMCwwLCJwX3syM31eKlxcTGNhbFxcb3RpbWVzX3tWXzN9cF97MTJ9XipcXExjYWxcXG90aW1lc197Vl8zfXBfezF9XipcXE1jYWwiXSxbMSwwLCJwX3syM31eKlxcTGNhbFxcb3RpbWVzX3tWXzN9cF97Mn1eKlxcTWNhbCJdLFsxLDEsInBfM1xcTWNhbCJdLFswLDEsInBfezEzfV4qXFxMY2FsXFxvdGltZXNfe1ZfM31wX3sxfV4qXFxNY2FsIl0sWzAsMSwiMVxcb3RpbWVzIHBfezEyfV4qXFxnYW1tYSJdLFsxLDIsInBfezIzfV4qXFxnYW1tYSJdLFswLDMsIlxccGhpXFxvdGltZXMgMSIsMl0sWzMsMiwicF97MTN9XipcXGdhbW1hIiwyXV0=
		\begin{tikzcd}[column sep=3.5em, row sep=2.5em]
			{p_{23}^*\Lcal\otimes_{V^3}p_{12}^*\Lcal\otimes_{V^3}p_{1}^*\Mcal} & {p_{23}^*\Lcal\otimes_{V^3}p_{2}^*\Mcal} \\
			{p_{13}^*\Lcal\otimes_{V^3}p_{1}^*\Mcal} & {p_3^*\Mcal}\rlap{ .}
			\arrow["{1\otimes p_{12}^*\gamma}", from=1-1, to=1-2]
			\arrow["{p_{23}^*\gamma}", from=1-2, to=2-2]
			\arrow["{\phi\otimes 1}"', from=1-1, to=2-1]
			\arrow["{p_{13}^*\gamma}"', from=2-1, to=2-2]
		\end{tikzcd}
		\]
		The data $(V\to \Xcal)$, $\Mcal$, $\Lcal$, $\gamma$ and $\phi$ reconstructs the order $\Acal$.
		
		We can blow-up $V$ at $\Mcal$ to obtain $f\colon V'\coloneq\Bl_{\Mcal}(V)\to V$, and by \Cref{lem: compat flat smooth cov} (and importantly as twisting by a line bundle does not change the blow-up) we obtain a diagram 
		\begin{equation}\label{eq: groupoid morph}
			\begin{tikzcd}
				V^{\prime,2} & {\Bl_{\Mcal}(V)} &  \\
				V^2={V\times_{\Xcal} V} & V & \Xcal
				%				\arrow["\pi", from=1-3, to=2-3]
				%				\arrow["b", from=1-2, to=1-3]
				\arrow[from=2-2, to=2-3]
				\arrow["f", from=1-2, to=2-2]
				\arrow["p_2"', shift right=0.75, from=2-1, to=2-2]
				\arrow["p_1", shift left=0.75, from=2-1, to=2-2]
				\arrow["g"', from=1-1, to=2-1]
				\arrow["q_2"', shift right=0.75, from=1-1, to=1-2]
				\arrow["q_1", shift left=0.75, from=1-1, to=1-2]
				\arrow["\lrcorner"{anchor=center, pos=0.125}, draw=none, from=1-1, to=2-2]
			\end{tikzcd}
		\end{equation}
		where
		\[
		\begin{tikzcd}[ampersand replacement=\&, row sep= 0.6em, column sep= -0.75em]
			{V^{\prime,2} } \& {\Bl_{p_1^*\Mcal}(V^2)} \& {\Bl_{\Lcal\otimes_{V^2}p_1^*\Mcal}(V^2)} \& {\Bl_{p_2^*\Mcal}(V^2)} \& \\
			\& {V^2\times_{p_1,V} \Bl_{\Mcal}(V)} \&\& {V^2\times_{p_2,V} \Bl_{\Mcal}(V)}\rlap{ .} \&
			\arrow[draw=none, from=1-1, to=1-2, "\coloneq" marking]
			\arrow[equal, from=1-2, to=1-3]
			\arrow[equal, from=1-3, to=1-4]
			\arrow[equal, from=1-2, to=2-2]
			\arrow[equal, from=1-4, to=2-4]
			\arrow[equal, from=1-4, to=2-4]
		\end{tikzcd}
		\]
		In particular, as $ {V^2\times_{p_1,V} \Bl_{\Mcal}(V)}= {V^2\times_{p_2,V} \Bl_{\Mcal}(V)}$, we naturally have that $V^{\prime,2}\rightrightarrows V'$ is an \'etale groupoid and the left part of \eqref{eq: groupoid morph} gives a morphism of groupoids.
		The isomorphism $\gamma$ gives rise to an isomorphism
		\[
		\gamma'\colon g^*\Lcal\otimes_{V^{\prime,2}}q_1^*f^\flat\Mcal\isoto q_2^*f^\flat\Mcal
		\]
		as
		\begin{align*}
			q_2^*f^\flat\Mcal = g^\flat p_2^*\Mcal = g^\flat ( \Lcal\otimes_{V^2}p_1^* \Mcal) = g^*\Lcal\otimes_{V^{\prime,2}}g^\flat p_1^* \Mcal = g^*\Lcal\otimes_{V^{\prime,2}}q_1^* f^\flat\Mcal. 
		\end{align*}
		This gives us an isomorphism $q^*_1{\sEnd}(f^\flat\Mcal)\isoto q^*_2{\sEnd}(f^\flat\Mcal)$ satisfying the cocycle condition.
		
		Putting $\Ycal\coloneq[V'/V^{\prime,2}]$ we obtain a (not-necessarily smooth any more) Deligne--Mumford stack $\Ycal$ together with a coherent sheaf of algebras $\Bcal$ over $\Ycal$.
		Furthermore, by construction, this comes equipped with an \'etale cover $b\colon V'\to \Ycal$ with $b^*\Bcal=\sEnd(f^\flat\Mcal)$, i.e.\ $\Bcal$ is Azumaya, and a morphism $\pi\colon \Ycal\to \Xcal$ which is locally a blow-up
		% https://q.uiver.app/#q=WzAsNCxbMSwwLCJcXFljYWwiXSxbMSwxLCJcXFhjYWwiXSxbMCwwLCJcXEJsX3tcXE1jYWx9KFYpIl0sWzAsMSwiViJdLFswLDEsIlxccGkiXSxbMiwwLCJiIl0sWzMsMV0sWzIsMywiZiJdXQ==
		\[
		\begin{tikzcd}[ampersand replacement=\&]
			{\Bl_{\Mcal}(V)} \& \Ycal \\
			V \& \Xcal
			\arrow["\pi", from=1-2, to=2-2]
			\arrow["b", from=1-1, to=1-2]
			\arrow[from=2-1, to=2-2]
			\arrow["f"', from=1-1, to=2-1]
			\arrow["\lrcorner"{anchor=center, pos=0.125}, draw=none, from=1-1, to=2-2]
		\end{tikzcd}
		\]	
		(use, e.g., \cite[\href{https://stacks.math.columbia.edu/tag/04ZN}{Tag 04ZN}]{stacks-project}).
		In addition, we have a morphism $\pi^*\Acal\to\Bcal$ given locally by
		\[
		f^*{\sEnd}(\Mcal)\to \sEnd(f^*\Mcal) \to \sEnd(f^\flat\Mcal),
		\]
		where the last morphism is given by first post-composing 
		\[
		\sHom(f^*\Mcal,f^*\Mcal)\to \sHom(f^*\Mcal,f^\flat\Mcal)
		\] 
		and noting that $\sHom(f^*\Mcal,f^\flat\Mcal)=\sHom(f^\flat\Mcal,f^\flat\Mcal)$.
		It is clear that $\pi^*\Acal\to\Bcal$ is generically an isomorphism.
		
		To end this step, note that even though the morphism $\pi\colon \Ycal\to \Xcal$ is proper birational and representable by schemes, it need not be a blow-up as it is only so locally.		
		However, as blow-ups are `cofinal' \cite[Theorem 3.1]{Rydh} there exists a sequence of blow-ups $\Zcal\to \Ycal$ such that the composition $\Zcal\to \Xcal$ is also a composition of blow-ups.
		Therefore, by replacing $\Ycal$ by $\Zcal$, we can assume $\pi$ is a composition of blow-ups (and so the morphism to $X$ is a composition of stacky blow-ups as we have a root stack from Step 2).
		
		\subsection*{Step 4}
		\refstepcounter{Step}\label[Step]{step4}
		We have now reduced to the case where the order $\Acal$ is Azumaya and $\pi$ is a composition of stacky blow-ups, but the Deligne--Mumford stack $\Xcal$ is not-necessarily smooth any more.
	    Consequently, taking a final resolution of singularities of $\Xcal$ (by blow-ups), e.g.\ these exist as smooth functorial resolutions exist in characteristic zero and by cofinality of blow-ups, and replacing $\Acal$ by its pullback finishes the proof of \Cref{thm: azumification}.	
		
		\begin{remark}
			By further stacky blow-ups one can achieve even more favourable situation, e.g.\ using the destackifications from \cite{Bergh,BerghRydh} one can ensure that the induced morphism on the coarse space $\Xcal_{\mathsf{cs}}\to X$ is a resolution of singularities (by a sequence of blow-ups) and that $\Xcal\to\Xcal_{\mathsf{cs}}$ is a root stack (in an snc divisor).
		\end{remark}		
		
		\section{A comparison to the resolutions of Chan--Ingalls 
%			\texorpdfstring{\cite{Chan/Ingalls:2005}}{[CI05]}
		}\label{sec:comparing_to_CI}
		
		The Azumayafication procedure of \Cref{thm: azumification} is simpler when starting from an order over a surface.
		Over a smooth (potentially stacky) surface any reflexive module is locally free; see e.g.\ 		\cite[\href{https://stacks.math.columbia.edu/tag/0B3N}{Tag 0B3N}]{stacks-project}.
		Consequently, the reflexive Azumaya algebra obtained in \Cref{step2} is already Azumaya and so \Cref{step3,step4} are not necessary.
		
		This observation allows us to make some comparisons with the resolutions constructed in Chan--Ingalls  \cite{Chan/Ingalls:2005}.
		The authors proved that any order over a surface can be resolved by a terminal order.
		An order $\Acal$ over a surface $S$ is called \emph{terminal} if 
		\begin{itemize}[leftmargin=20pt]
			\item the surface $S$ is smooth,
			\item the order $\Acal$ is normal, i.e.\ reflexive and for any codimension one point $s\in S$ the radical $\rad \Acal_s$ is principal as a left and right ideal,
			\item the ramification divisor has normal crossings,
			\item extra conditions on the `secondary ramification' at the nodes of the ramification divisor. 
			These conditions are not so relevant for our discussion here, so we do not spell them out; they can be found in \cite[Definition 2.5]{Chan/Ingalls:2005}.
		\end{itemize}
		Terminal orders always have global dimension two.
		Moreover, it is shown in \cite[Corollary 4.5]{Chan/Ingalls:2005} that sufficiently blowing up the nodes of the ramification divisor yields a resolution by a terminal order.
		In fact, by further blowing up, one may assume that the terminal order has a snc ramification divisor.
		As terminal orders generally still have ramification, they need not be Azumaya.
		However, when the ramification divisor has simple normal crossings, one can obtain an Azumaya algebra over a root stack using \Cref{prop:rooting}.
		The purpose of this short section is to show that this Azumaya algebra and the terminal order are Morita equivalent.
		
		We start with a slightly more general proposition; it is similar to \cite[Theorem 4.5]{Faber/Ingalls/Okawa/Satriano}, although the proof below is different in places. 
		
		\begin{proposition}\label{prop:Morita_in_dim2}
			Let $\Acal$ be a tame order of global dimension two over a normal surface $S$ with effective ramification divisor. 
			Furthermore, let $\Bcal$ be the reflexive Azumaya algebra over the root stack $\pi\colon \Scal\to S$ from \Cref{prop:rooting}.
			Then the morphism $\pi$ induces an equivalence of categories of quasi-coherent sheaves
			\[
				\Qcoh(\Scal,\Bcal)=\Qcoh(S,\Acal).
			\]
		\end{proposition}
		\begin{proof}
			As noted in \Cref{step2} of the proof of \Cref{thm: azumification} there is a natural morphism $\pi^*\Acal\to \Bcal$ induced by multiplication.
			This, together with the morphism $\pi$, induces an adjunction
			\begin{equation}\label{eq:adjunction_global}
				\begin{tikzcd}
					\Qcoh(\Scal,\Bcal)\\
					\Qcoh(S,\Acal)\rlap{.}
					\arrow[from=1-1, to=2-1, "{\pi_*}",""{name=0, anchor=center, inner sep=0}, bend left =45]
					\arrow[from=2-1, to=1-1, "{\pi^*(-)\otimes_{\pi^*\Acal}\Bcal}",""{name=1, anchor=center, inner sep=0},, bend left =45]
					\arrow["\dashv"{anchor=center}, draw=none, from=0, to=1]
				\end{tikzcd}
			\end{equation}

			We claim that this adjunction is the sought equivalence. 
			For this we need to check that the (co)unit of the adjunction are isomorphisms.
			We can do this complete locally on $S$ by faithfully flat descent.
			Therefore, we may assume $S=\Spec R$ where $R$ is a two dimensional complete local Noetherian ring, 
			$\Acal=\Lambda$ with $\operatorname{gl.dim}\Lambda=2$ and $\Scal=\left[\Spec \widetilde{R} /\GG_m^n\right]$ where $ \widetilde{R}$ is the centre of the $\ZZ^n$-graded algebra $\widetilde{\Lambda}$ corresponding to $\Bcal$.
			The equivalences 
			\[
				\Qcoh\left(\left[\Spec \widetilde{R} /\GG_m^n\right],\Bcal\right)= \grMod_{\ZZ^n} \widetilde{\Lambda}
				\quad \text{and} \quad
				\Qcoh(S,\Acal) = \Mod \Lambda,
			\]
			translate the adjunction \eqref{eq:adjunction_global} into
			\begin{equation}\label{eq:adjunction_local}
				\begin{tikzcd}
					\grMod_{\ZZ^n} \widetilde{\Lambda} \\
					 \Mod \Lambda \rlap{,}
					\arrow[from=1-1, to=2-1, "{(-)_0}",""{name=0, anchor=center, inner sep=0}, bend left =45]
					\arrow[from=2-1, to=1-1, "{-\otimes_{\Lambda}\widetilde{\Lambda}}",""{name=1, anchor=center, inner sep=0},, bend left =45]
					\arrow["\dashv"{anchor=center}, draw=none, from=0, to=1]
				\end{tikzcd}
			\end{equation}
			where $(-)_0$ indicates taking the graded piece of degree the neutral element $0\in \ZZ^n$; note that $(\widetilde{\Lambda})_0=\Lambda$.
			
			That the adjunction \eqref{eq:adjunction_local} is an equivalence is equivalent to the  $\ZZ^n$-graded algebra $\widetilde{\Lambda}$ being strongly graded; see, e.g., \cite[Theorem I.3.4]{Nastasescu/VanOystaeyen}.
			The crucial observation now is that by \cite[Theorem 1.10]{Ramras} or \cite[Proposition 2.17]{IyamaWemyss2014a}, as $\Lambda$ has global dimension two, any reflexive $\Lambda$-module is projective (as $\Lambda$-module).
			Consequently, for any two divisorial ideals $I$ and $J$ it follows that $IJ$ is reflexive; as an $R$-module $IJ$ is a summand of some power of $I$ (or $J$ if the reader prefers).
			Thus 
			\[
				I\cdot J = (IJ)^{**} =IJ,
			\] 
			and so $\widetilde{\Lambda}$ is clearly strongly graded, finishing the proof.
		\end{proof}
		\begin{corollary}
			Let $(S,\Acal)$ be a terminal order with simple normal crossing ramification divisor. 
			Then the reflexive Azumaya algebra~$\Bcal$ over the root stack $\pi\colon \Scal\to S$ from \Cref{prop:rooting} is Azumaya and $\Scal$ is smooth.
			Furthermore, the morphism $\pi$ induces an equivalence $\Qcoh(\Scal,\Bcal)=\Qcoh(S,\Acal)$.
		\end{corollary}
		\begin{proof}
			First note that $\Scal$ is smooth as $S$ is smooth and $\Acal$ is assumed to have simple normal crossing ramification divisor.
			Consequently, $\Bcal$ is Azumaya as any reflexive $\Ocal_{\Scal}$-module, for $\Scal$ smooth, is locally free.
			The stated Morita equivalence follows immediately from \Cref{prop:Morita_in_dim2} as terminal orders have global dimension two.
		\end{proof}
		\begin{remark}
			In \cite{Chan/Ingalls:2005} the characteristic of $\kk$ is allowed to be positive, so long as it is coprime to the degree of the order. 
			The two results above work in this generality, the existence of  `unramified splitting fields' of the function field of the order makes \Cref{prop:rooting} work in this generality.
		\end{remark}
		
		The ramification divisor of a terminal order need not always have simple normal crossings.	
		\begin{example}[Terminal order with nodal ramification]
			By the Artin--Mumford sequence \cite[Theorem 1]{Artin/Mumford} %, see also \cite[\S2.5]{Artin/deJong}.
			there exists a terminal order over $\Abb^2_\kk$ ramified along the nodal curve $\VV(y^2-x^3-x^2)$.
			Although \Cref{prop:Morita_in_dim2} can be applied to this order; the root stack need not be smooth nor does the reflexive Azumaya algebra over it need to be Azumaya.
			For example, \'{e}tale locally the root stack looks like $\Spec k[x,y,t]/(t^2-xy)$ which is singular at the origin.
		\end{example}
		
		To end, let us observe that both approaches diverge in general; as is to be expected there are multiple ways to resolve singularities. 
		A good illustration is given by looking at a specific example of a canonical order.
		\begin{example}[Canonical order ramified on two intersecting lines]
			The Artin--Mumford sequence guarantees the existence of an order over $\Abb^2_\kk$ ramified along $\VV(xy)$ such that neither $\VV(x)$ nor $\VV(y)$ is totally ramified at the origin; this is an example of a non-terminal canonical order. 
			One way to resolve this order, as in \cite{Chan/Ingalls:2005}, is by blowing up the origin and thus obtaining a resolution via a terminal order. 
			Alternatively, one can use \Cref{prop:rooting}; as the ramification divisor of the order has simple normal crossings the root stack is smooth and so one obtains an Azumaya algebra.
		\end{example}
				
		\bibliographystyle{amsalpha}
		\bibliography{bibliography}
	\end{document}